\documentclass{aptpub}
\authornames{Hugo Magaldi}
\shorttitle{The stochastic Bessel operator at high temperatures}

\usepackage[T1]{fontenc}
\usepackage{epsfig}
\usepackage{graphicx}
\usepackage{amssymb}
\usepackage{amsmath}
\usepackage[utf8]{inputenc}
\usepackage{amsfonts}
\usepackage{float}
\usepackage{hyperref}
\usepackage{natbib}
\usepackage[english]{babel}
\usepackage[toc,page]{appendix}
\usepackage{appendix}

\usepackage{multirow}
\usepackage{lscape}
\usepackage{array}
\newcolumntype{M}[1]{>{\centering\arraybackslash}m{#1}}
\usepackage{natbib}
\usepackage{mathdots}
\setcitestyle{numbers}
\usepackage{natbib, hyperref}
\setcitestyle{square}
\usepackage{epigraph}
\usepackage{chngcntr} 
\usepackage{graphicx}
\graphicspath{{figures/}}
\usepackage{enumitem}
\usepackage{indentfirst}
\usepackage{empheq}
\usepackage[most]{tcolorbox}
\usepackage{afterpage}
\usepackage{color}
\usepackage[margin=1.2in,headheight=13.6pt]{geometry}
\usepackage{enumitem}
\usepackage{stmaryrd}

\newcommand{\eps}{\varepsilon}

\newcommand{\cD}{\mathcal{D}}

\newcommand{\gl}{\lambda}
\newcommand{\gL}{\Lambda}

\begin{document}

\title{The Stochastic Bessel operator at high temperatures}\\
%\title{The Stochastic Bessel operator at high temperatures}
\begin{center}
\author[CEREMADE, Université Paris Dauphine - PSL]{Hugo Magaldi}\address{Université Paris Dauphine - PSL, Place du Maréchal de Lattre de Tassigny, 75016 Paris, France}
\end{center}

%\input{./title.tex}

%%%%%%%%%%%%%%%%
%%%%%%%%%%%%%%%% DEBUT DU TEXTE
%%%%%%%%%%%%%%%% 

\begin{abstract}
We know from Ram\'irez and Rider \cite{DiffusionAtTheRandomMatrixHardEdge} that the hard edge of the spectrum of the $\beta$-Laguerre ensemble converges, in the high-dimensional limit, to the bottom of the spectrum of the stochastic Bessel operator. Using stochastic analysis techniques, we show that, in the high temperatures limit, the rescaled eigenvalues point process of the stochastic Bessel operator converges to a limiting point process characterized with coupled stochastic differential equations.
\end{abstract}

\medskip


\medskip

% APT \ams{60H25}{60H10; 60H30}

\section{Context and motivation} \label{sec_intro}

Coulomb gases, also known as $\beta$-ensembles, are probability measures on sets of points. In statistical physics, these points are seen as elementary particles on the real line confined under a random potential while repelling each other. The Dyson parameter $\beta$ acts as an inverse temperature and can affect the shape of the potential and the strength of the repulsive force. When $\beta$ is close to $\infty$, the system becomes ordered due to high repulsion between the particles. Conversely, when $\beta$ is close to $0$, weak repulsion leads to a disordered system. In recent years, $\beta$-ensembles and their connections to stochastic operators have attracted significant interest (see \cite{MatrixModelsForBetaEnsembles}, \cite{GlobalSpectrumFluctuations}, \cite{GapUniversalityOfGeneralizedWignerAndBetaEnsembles},
\cite{BetaEnsemblesStochasticAirySpectrum},
\cite{BulkUniversalityOfGeneralBetaEnsemblesWithNonConvexPotential},
\cite{UniversalityOfGeneralBetaEnsembles},  \cite{OrthogonalAndSymplecticMatrixModels}, \cite{LocalSemicircleLawInTheBulkForGaussianBetaEnsemble}, \cite{ContinuumLimitsOfRandomMatrices}, \cite{LocalSemicircleLawAtTheSpectralEdgeForGaussianBetaEnsembles}). 

\medskip

\noindent The size $n$ $\beta$-Laguerre ensemble is a two-parameter family of distributions on $(\mathbb{R}_+^{*})^n$ with density function with respect to the Lebesgue measure
\begin{align}\label{densityLaguerre}
\mathcal{P}_n^{\beta, a} \big(\gl^{(n)}_0 \leqslant \gl^{(n)}_1 \ldots \leqslant \gl^{(n)}_{n-1}\big)= \frac{1}{Z_{n}^{\beta,a}} \prod_{i < j} |\gl^{(n)}_i - \gl^{(n)}_j|^\beta \times \prod_{k=0}^{n-1} (\gl^{(n)}_k)^{\frac{\beta}{2}(a+1) -1} e^{-\frac{\beta}{2} \gl^{(n)}_k} 1_{\gl^{(n)}_k > 0}.
\end{align}
The parameters satisfy $\beta>0$ and $a > -1$, and $Z_{n}^{\beta,a}$ is an explicitely computable normalizing constant. When $\beta=1$ (resp. $\beta=2$, $\beta=4$), $\mathcal{P}_n^{\beta, a}$ is the density of the eigenvalues of a real (resp. complex, unitary) Wishart matrix of size $(n, n)$. This connection to random matrices was extended to any $\beta>0$ by Dimitriu and Edelman \cite[Theorem 3.4]{MatrixModelsForBetaEnsembles}, adapting earlier work from Silverstein \cite{TheSmallestEigenvalueOfALargeDimensionalWishartMatrix}.  They found a set of bidiagonal random matrices $\big(L_{n}^\beta\big)_{n>0}$ of size $(n,n+a)$ such that $\mathcal{P}_n^{\beta, a}$ is the density of the eigenvalues of $L_{n}^\beta(L_{n}^\beta)^T$.

\medskip

\noindent When $n$ tends to infinity, the rescaled empirical measure of the eigenvalues converges weakly a.s. to the Marchenko-Pastur distribution:
\begin{equation*}
\mu(\mathrm{d} x) = \frac{1}{2\pi x}\sqrt{x(4 -x)}\mathbf{1}_{[0,2]}(x)\mathrm{d} x.
\end{equation*}
We refer to \cite{OnSpectralMeasuresOfRandomJacobiMatrices} for a review of the asymptotic global statistics of classical $\beta$-ensembles.
 
\medskip

\noindent Because of the positivity constraint for the eigenvalues of $L_{n}^\beta(L_{n}^\beta)^T$, $0$ is called a \emph{hard edge} for the $\beta$-Laguerre ensemble. To study local statistics at the hard edge in the high-dimensional limit, Ram\'irez and Rider \cite{DiffusionAtTheRandomMatrixHardEdge} introduced a stochastic operator as limiting object for the $\beta$-Laguerre ensemble.

\begin{definition}[The stochastic Bessel operator] \label{def_3_SBO} ~\\ \emph{
Let $B$ be a standard Brownian motion on $\mathbb{R}_+$. For $\beta>0$ and $a>-1$, the stochastic Bessel operator (known as SBO) is the random differential operator
\begin{gather*}
\mathfrak{G}^{\beta,a} = -\frac{1}{m(x)}\frac{\mathrm{d}}{\mathrm{d} x}\Big(\frac{1}{s(x)}\frac{\mathrm{d}}{\mathrm{d} x}\Big), \\ 
m(x)=\exp\big(-(a+1)x-\frac{2}{\sqrt{\beta}}B(x)\big), \quad s(x)=\exp\big(ax+\frac{2}{\sqrt{\beta}}B(x)\big),
\end{gather*}
defined on a subset of $L^2(\mathbb{R}_+,m)$ with Dirichlet and Neumann boundary conditions at $0$ and at infinity, respectively.
}\end{definition}

\noindent They showed a connection between the eigenvalues of the $\beta$-Laguerre ensemble at the hard edge and the low-lying eigenvalues of the SBO:

\begin{proposition}[Limit of the $\beta$-Laguerre ensemble at the hard edge \protect{\cite[Theorem 1]{DiffusionAtTheRandomMatrixHardEdge}}] \label{prop_SBOspec}~\\
With probability one, when restricted to the positive half-line with Dirichlet
conditions at the origin, $\mathfrak{G}^{\beta,a}$ has a discrete spectrum with single eigenvalues $0<\Lambda^{\beta,a}(0)<\Lambda^{\beta,a}(1)<\ldots \uparrow \infty$. Moreover, with $0<\gl^{(n)}_0 < \gl^{(n)}_1 < \ldots < \gl^{(n)}_n$ the ordered points of the $\beta$-Laguerre ensemble of size $n$,
\begin{equation*}
\big\lbrace n\gl^{(n)}_0, n\gl^{(n)}_1, \ldots, n\gl^{(n)}_k\big\rbrace \Longrightarrow \big\lbrace \Lambda^{\beta,a}(0)<\Lambda^{\beta,a}(1)<\ldots <\Lambda^{\beta,a}(k)\big\rbrace,
\end{equation*}
(jointly in law) for any fixed $k<\infty$ as $n \uparrow \infty$.
\end{proposition}

\noindent In this paper, we study the convergence of the lowest eigenvalues of the SBO in the high temperatures limit, when the inverse temperature parameter $\beta$ tends to $0$.

\section{Our result}

When $\beta$ tends to $0$, the smallest eigenvalues of the SBO get close to the hard edge at $0$ at an exponential rate. In order to get a non trivial limit, we therefore consider the rescaled eigenvalues
\begin{align} \label{eq_eig_resc}
\mu^{\beta}(k) := \beta \ln\big(1/\Lambda^{\beta,a}(k)\big), \ k \geqslant 0. 
\end{align}
Note that we thus reverse the ordering of the eigenvalues, and that any $\mu^{\beta}(k)$ is positive for $\beta$ small enough. We restrict ourselves to the case $a>0$ because the estimations computed in Section \ref{sec_6_expl} and Section \ref{sec_6_tight} differ if $a \in (-1,0]$. Since the parameter $a$ is set, we omit it from our notations going forward.

\begin{theorem}[Convergence of the low-lying eigenvalues of the SBO]\label{thm_main}
When $\beta$ tends to $0$, the rescaled eigenvalues point process of the SBO $(\mu^{\beta}(k),\; k \geqslant 0)$ converges in law towards a random simple point process on $\mathbb{R}_+$ which can be described using coupled SDEs.
\end{theorem}

\noindent  The convergence holds for a well chosen topology of measures on $\mathbb{R}_+$, corresponding to a left-vague/right-weak topology (see after Proposition \ref{prop_cv_eig} for more details). The limiting point process is simple in the sense that all its points are distinct almost surely. We will characterize it (similarly to the SBO eigenvalues) through the coupled diffusions (\ref{SDEq+}) and (\ref{SDEq-}).

\medskip

\noindent  Usually, one expects that, when the temperature is high, the limiting point process is no longer repulsive and corresponds to a Poisson point process as the noise becomes dominant. Here, we get a different result. It comes from the competition between the strong repulsive interaction and attraction at the hard edge for small $\beta$ (see \eqref{densityLaguerre}). Because of this interaction, the repulsive factor does not disappear in the limit.

\section{Strategy of proof and limiting point process} \label{sec_strat}

\subsection{Riccati transform}

If $\psi$ solves the eigenvalues equation $\mathfrak{G}^{\beta,a} \psi=\Lambda^{\beta,a} \psi$ with initial conditions $\psi(0)=c_{0}$, $\psi^{\prime}(0)=c_{1}$, then Dumaz, Li and Valk\'o \cite[Proposition 7]{OperatorLevelHardToSoft} showed that $\left(\psi, \psi^{\prime}\right)$ is the unique strong solution of the stochastic differential equation system:
\begin{equation*}
\mathrm{d} \psi(x)=\psi^{\prime}(x) \mathrm{d} x, \quad \mathrm{d} \psi^{\prime}(x)=\frac{2}{\sqrt{\beta}} \psi^{\prime}(x) \mathrm{d} B(x)+\left((a+\frac{2}{\beta}) \psi^{\prime}(x)-\Lambda^{\beta,a} e^{-x} \psi(x)\right) \mathrm{d} x,
\end{equation*}
with the corresponding initial conditions, and $B$ the Brownian motion from Definition \ref{def_3_SBO}.

\noindent Using the Riccati transform $p=\psi^{'}/ \psi$ from Halperin  \cite{GreenFunctionsForAParticle} defined away from the zeros of $\psi$, Ram\'irez and Rider \cite{DiffusionAtTheRandomMatrixHardEdge} introduced the family of coupled diffusions $\big(p^{\beta}_\gl,\; \gl \in \mathbb{R}_+^*\big)$:

\begin{align} \label{eq_p}
	\mathrm{d} p^{\beta}_\gl(t) = \frac{2}{\sqrt{\beta}} p^{\beta}_\gl(t) \mathrm{d} B(t) + \big((a+\frac{2}{\beta})p^{\beta}_\gl(t) - p^{\beta}_{\gl}(t)^2 - \gl e^{-t}\big) \mathrm{d} t.
\end{align}
\noindent The choice $\psi(0)=0, \psi^{'}(0)=1$ provides the initial condition $p^{\beta}_\gl(0)=\infty$. The diffusion $p^{\beta}_\gl$ may explode to $-\infty$ if $\psi$ hits $0$, in which case it immediately restarts from $\infty$. Indeed, at the first hitting time $\tau$ of $0$ by $\psi$, we have $\psi^{'}(\tau)<0$ ($\psi^{'}(\tau)\leqslant 0$ since $\psi$ is positive near $0$ thus decreasing right before $\tau$, and $\psi^{'}(\tau) \neq 0$ since $\psi(\tau)=0$ and $\psi$ is not the null solution) and $\psi(t)<0$ after $\tau$. The same analysis can be conducted for potential successive hitting times of $0$. 

\noindent Ram\'irez and Rider proved the following equality in law that connects the family $\big(p^{\beta}_\gl,\; \gl \in \mathbb{R}_+^*\big)$ to the eigenvalues of  $\mathfrak{G}^{\beta,a}$:
\begin{equation}\label{eq_RR}
\forall k \in \mathbb{N}, \ \big\lbrace p_\lambda^\beta\mbox{ explodes at most } k \mbox{ times on } (0,\infty)\big\rbrace = \big\lbrace \Lambda^{\beta,a}(k)>\lambda\big\rbrace.
\end{equation}

%\noindent Plainly said, the number of explosions of $p^{\beta}_\gl$ on $(0,\infty)$ is the number of eigenvalues of $\mathfrak{G}^{\beta,a}$ below $\gl$.

\noindent It is crucial to note that the \emph{same Brownian motion} $B$ from Definition \ref{def_3_SBO} drives the whole family of SDEs in (\ref{eq_p}). It implies important properties such as the monotonicity of the number of explosions of $p^{\beta}_\gl$ (which turns out to be finite). In fact, the number of explosions of $p^{\beta}_\gl$ on $(0,\infty)$ is the number of eigenvalues of $\mathfrak{G}^{\beta,a}$ below $\gl$.

%the number of explosions of $p^{\beta}_\gl$ over $\mathbb{R}_+$ corresponds to $N^{\beta}_\gl$, the counting function of the eigenvalues of the SBO. 

\subsection{Rescaled diffusions} We study the small beta limit of the family of diffusions $(p^{\beta}_\gl)$ from (\ref{eq_p}) when $\gl$ is properly rescaled with $\beta$, i.e. when $\gl$ is such that $\beta \ln (1/\gl)$ is of order $1$. 
 
\noindent  Notice that, when $p^{\beta}_\gl$ reaches $0$, the random term vanishes and the diffusion drift is negative. It implies that $p^{\beta}_\gl$ never reaches $0$ from below. It is easy to check that the hitting times of $0$ form a discrete point process.

\noindent Let us fix $\mu >0$ and set $\gL^\beta := \exp(-\mu/\beta)$. Using the property above, we define the diffusion $q^\beta_\mu(t)$, which equals 
\begin{align*}
\mathbb{R} \ni  q^+_\mu(t) &:= \beta \ln \big(p^{\beta}_{\gL^\beta}(t/(4 \beta)) \big) \text{ when } p^{\beta}_{\gL^\beta}(t/(4\beta)) >0, \\ 
\mathbb{R} \ni q^-_\mu(t) &:= -\beta \ln\big(-p^{\beta}_{\gL^\beta}(t/(4\beta))\big) - \mu - t/4 \text{ when } p^{\beta}_{\gL^\beta}(t/(4\beta)) <0, 
\end{align*}

\noindent Using the Itô formula $d\ln\big(Z(t)\big) = dZ(t)/Z(t) -d\langle Z \rangle_t /Z^2(t)$ for $Z(t)$ a positive Itô process, we can show that the diffusions $q^+_\mu(t)$ and $q^-_\mu(t)$ follow the following SDEs:
\begin{align}
	\mathrm{d} q^+_\mu &= \mathrm{d} W(t) + \frac{1}{4} \Big(a - \exp(q^+_\mu(t)/\beta) - \exp(-(q^+_\mu(t) + t/4 + \mu)/\beta) \Big) \mathrm{d} t, \label{SDEq+}\\
	\mathrm{d} q^-_\mu &= \mathrm{d} W(t) + \frac{1}{4} \Big(-(a+1) - \exp(q^-_\mu(t)/\beta) - \exp(-(q^-_\mu(t) + t/4 + \mu)/\beta) \Big) \mathrm{d} t, \label{SDEq-}
\end{align}
where $W$ is a Brownian motion corresponding to different rescalings of the initial Brownian motion $B$. These rescalings, including the new magnitude $t/(4\beta)$, are chosen so that the Brownian motion $W$ stands on its own, allowing an easier comparison between the terms of the drift in the $\beta \rightarrow 0$ limit.

\noindent The diffusions $q_\mu^\pm$ may explode to $-\infty$ in a finite time. By definition, the diffusion $q^{\beta}_\mu$ alternates between $q^+_\mu$ and $q^-_\mu$: it starts to follow $q_\mu^+$ and each time $q^{\beta}_\mu =q_\mu^+$ (resp. $q_\mu^-$) reaches $-\infty$, $q^{\beta}_\mu$ immediately restarts from $\infty$ and follows $q_\mu^-$ (resp. $q_\mu^+$). This alternating system is a consequence of the logarithmic transformation, which leads to a change of diffusion each time $p^{\beta}_{\gL^\beta}$ changes sign. \\
Let us define the critical line:
\begin{equation}\label{criticalline}
	c_{\mu}(t) := -\mu - t/4\,.
\end{equation}

\noindent This definition of $c_\mu$ makes sense in light of the last exponential term from (\ref{SDEq+}) and (\ref{SDEq-}). The positions of $q^+_\mu$ and $q^-_\mu$ with respect to $c_\mu$ condition the order of magnitude of this term and wether it prevails or not over the other terms of the drift.

\noindent Figure \ref{fig_6_typ} shows a sample path of the diffusion $q^{\beta}_\mu$. On this event, the diffusion $q^{\beta}_\mu$ explodes one time as $q_\mu^+$ (blue), then one time as $q_\mu^-$ (red), and then stays above the critical line $c_{\mu}(t)$ as $q_\mu^+$ (blue) and does not explode anymore.

%\hFig{Sample trajectory of the diffusion $q^{\beta}_\mu$.}
\begin{figure}[H] 
    %%\textbf{Your title}\par\medskip
    \centerline{\includegraphics[scale=0.2]{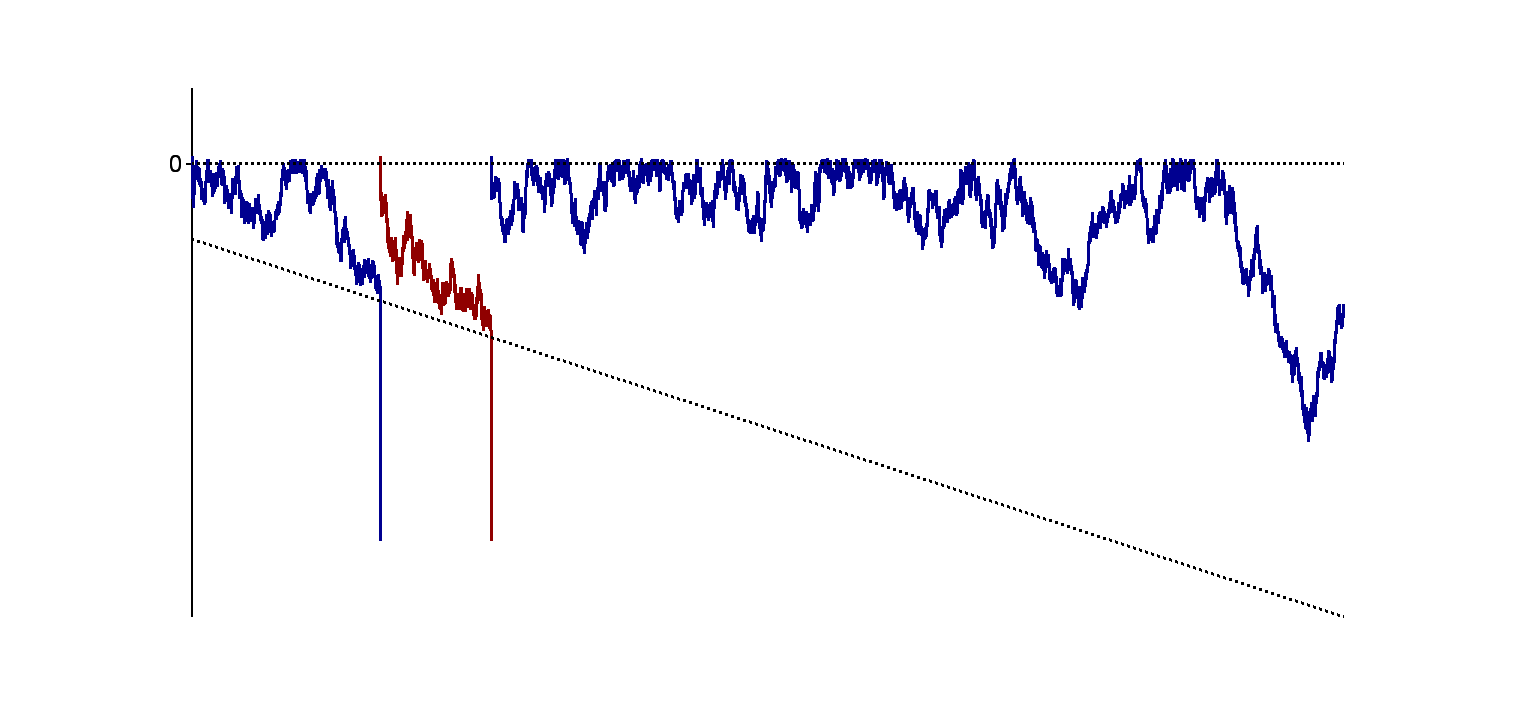}}
    \caption{Sample trajectory of the diffusion $q^{\beta}_\mu$.}
    \label{fig_6_typ}
\end{figure}

\noindent Roughly, the diffusion $q^+_\mu$ behaves as follows after each explosion time. First, it quickly goes down to values around $0$. Then, it spends some time between the line $t \mapsto c_{\mu}(t)$ and $0$, where it behaves as a reflected (downwards) Brownian motion with drift $a/4$. If it reaches the line $t \mapsto c_{\mu}(t)$ in a finite time then it quickly explodes to $-\infty$ after this hitting time.

\noindent  The behaviour of the diffusion $q^-_\mu$ is similar except that in the interval $[c_\mu(t),0]$, it behaves as a reflected (downwards) Brownian motion with drift $-(a+1)/4$. Therefore, it almost surely hits $c_{\mu}(t)$ when $a > 0$.

\noindent  There are two types of explosions for $q^{\beta}_\mu$: either $q^{\beta}_\mu$ explodes at a time $\xi^+$ such that $q^{\beta}_\mu(\xi^+) = q^+_\mu(\xi^+)$, which corresponds to the (rescaled) hitting times of $0$ by the initial diffusion $p^{\beta}_{\gL^\beta}$, or $q^{\beta}_\mu$ explodes at time $\xi^-$ such that $q^{\beta}_\mu(\xi^-) = q^-_\mu(\xi^-)$, in which case we get the (rescaled) explosion times of the initial diffusion $p^{\beta}_{\gL^\beta}$.

\noindent In the following, we denote by $\xi^+_\beta(0) < \xi^-_\beta(0) < \xi^+_\beta(1) < \xi^-_\beta(1) < \ldots$ the explosion times of the diffusion $q_\mu^\beta$ and by
\begin{align}\label{explosionsq-}
	\nu^\beta_\mu := \sum_{i \geqslant 0} \delta_{\xi^-_\beta(i)}\,
\end{align}
the measure corresponding to the (rescaled) explosions of $p^\beta_{\gL^\beta}$.

\noindent We will prove that, for a well-chosen topology, the trajectory of the diffusion $q^{\beta}_\mu$ converges in law, when $\beta$ tends to $0$, towards a non-trivial limit $r_\mu$, that we describe in the following paragraph.

\subsection{Description of the limiting point process} Let us now define $r_\mu$, the limiting diffusion of $q^{\beta}_\mu$, which will characterize the limiting point process from Theorem \ref{thm_main}. Its definition involves Brownian motions with drift reflected downwards at $0$. By definition, a Brownian motion with drift $\alpha$ reflected downwards at $0$ is a diffusion Markov process with infinitesimal operator $G \; :\; f \in \cD \mapsto \frac{1}{2} f'' + \alpha f'$ acting on the domain 
\begin{align*}
	\cD :=\Big\lbrace f \in C_b[0,\infty[, Gf \in C_b\big([0,\infty[\big),\; \lim_{x \downarrow 0} f'(x) = 0\Big\rbrace\,,
\end{align*}
where $C_b\big([0,\infty[\big)$ denotes the continuous and bounded functions on $[0,\infty[$.
Using the Skorohod problem, we can write this diffusion as
\begin{align*}
B(t) +\alpha t -  \sup_{s \leqslant t} \big(B(s) +\alpha s\big)\vee 0\,,
\end{align*}
where $B$ is a Brownian motion starting at $0$ or any negative point. 

\noindent To define the limiting diffusion $r_\mu$, we use \textit{the same Brownian motion} $W$ as in the diffusions $q_\mu^+$ and $q_\mu^-$ from (\ref{SDEq+}) and (\ref{SDEq-}). 

\begin{definition}[Limiting diffusion] \label{def_0_limdif} ~\\ \emph{
The limiting diffusion $r_\mu$ is defined as follows:}
\begin{itemize}
	\item \emph{It starts at $0$ at time $0$: $r_\mu(0) = 0$.}
	\item \emph{It then follows a Brownian motion with drift $a/4$  reflected downwards at $0$ (built from $W$)
\begin{align}\label{eq_r+}
	r^+:= W(t) + at/4 -  \sup_{s \leqslant t} \big(W(s) +as/4\big)\vee 0\,
\end{align}
until its first hitting time of the critical line $c_\mu(t)$ from \eqref{criticalline}.}
\item \emph{If it reaches $t \mapsto c_\mu(t)$ in a finite time, it immediately restarts at $0$ at this time, then follows a reflected downwards at $0$ Brownian motion with another drift $-(a+1)/4$ (also built from $W$) 
\begin{align}\label{eq_r-}
	r^-:= W(t) - (a+1)t/4 -  \sup_{s \leqslant t} \big(W(s) - (a+1)s/4\big)\vee 0\
\end{align}
until its first hitting time of the critical line $c_\mu(t)$ from \eqref{criticalline}.}
\end{itemize}
\emph{And so on, $r_\mu$ alternating between $r^+$ and $r^-$ each time it hits $t \mapsto c_\mu(t)$ and restarts at $0$.
Note that the probability that $r^+$ reaches the critical line decreases with time. On the other hand, since $a>0$, $r^-$ almost surely hits the critical line in a finite time.
}\end{definition}

\noindent Let $\xi^+_0(0) < \xi^-_0(0) < \xi^+_0(1) < \xi^-_0(1) < \ldots$ be the hitting times of the critical line $c_\mu$ by the diffusion $r_\mu$, where we omit to display the dependency on $\mu$ for lighter notations. We define the random measure associated to the point process $\big(\xi^-_0(k),\;k \geqslant 0\big)$:
\begin{align}\label{explosionsr}
	\nu^0_\mu := \sum_{i \geqslant 0} \delta_{\xi^-_0(i)}\,.
\end{align}

\noindent We then use the coupled random measures $\big(\nu^0_\mu,\; \mu >0\big)$ to define a discrete point process on $\mathbb{R}_+$. Since $\mu \in \mathbb{R}_+ \mapsto \nu_\mu(\mathbb{R}_+)$ decreases from $\infty$ to $0$ almost surely, it is easy to prove the following proposition:

\begin{proposition}[Limiting point process]
	There is a random variable $M_0$ valued in the Borel sets of $]0,\infty[$, such that, for all fixed $\mu_1 < \ldots < \mu_k$,
	\begin{align*}
		M_0[\mu_i,\infty[ = \nu^0_{\mu_i}(\mathbb{R}_+)\,.
	\end{align*}
Almost surely, the measure $M_0$ is unique, discrete, bounded from above and has an accumulation point at $0$.
\end{proposition}

\subsection{Strategy of the proof of Theorem \ref{thm_main}}

We can now state the desired convergence results towards the limiting measures which we need to prove Theorem \ref{thm_main}.

\begin{proposition}[Convergence of the explosion times of $p^\beta_{\gL^\beta}$]~\\\label{prop_cv_expl}
When $\beta$ tends to $0$, the measure $\nu^\beta_\mu$ (as in (\ref{explosionsq-})) converges to the measure $\nu^0_\mu$ (as in (\ref{explosionsr})), for the topology of weak convergence. 
\end{proposition}

\noindent It is immediate to extend this proposition for the joint law of $\nu^\beta_{\mu_i}$ when $\mu_1,\ldots,\mu_k$ are fixed positive numbers. It directly implies the following result on the finite dimensional laws of the point process $\big\lbrace\mu^{\beta}(i),\; i\geqslant 0\big\rbrace$, with $\mu^{\beta}(i)$ as in (\ref{eq_eig_resc}). Let us denote by $M_\beta$ the measure associated to this point process, i.e. $M_\beta := \sum_{i\geqslant 0} \delta_{\mu^{\beta}(i)}$.

\begin{proposition}[Convergence of the finite-marginals of the eigenvalue process] ~\\ \label{prop_cv_eig}
Fix $\mu_1 < \ldots < \mu_k$. When $\beta \to 0$, the random vector 
$\big(M_\beta [\mu_1,\infty[, \ldots,  M_\beta [\mu_k,\infty[\big)$ converges in law to the random vector $\big(M_0 [\mu_1,\infty[, \ldots,  M_0 [\mu_k,\infty[\big)$.
\end{proposition}

\noindent Using a similar reasoning as in \citep[Proof of Theorem $1$]{TheStochasticAiryOperatorAtLargeTemperature}, we consider the space of measures on $]0,\infty[$ with the topology that makes continuous the maps $m \mapsto \langle f, m \rangle$ for any continuous and bounded function $f$ with support bounded to the left. In other words, this is the vague
topology towards $0$ and the weak topology towards $\infty$. It allows to control the decreasing sequence of atom locations of $M_\beta$ from its first point.

\noindent Proposition \ref{prop_cv_eig} shows that the family of measures $(M_\beta)_{\beta >0}$ is tight (and therefore relatively compact by Prokhorov's theorem): indeed, the above convergence provides the required control on the mass given by $M_\beta$ to $[\mu, \infty[$ for any given $\mu>0$. Besides, this convergence uniquely identifies the finite-marginals of any limiting point. We thus deduce from Proposition \ref{prop_cv_eig} the convergence of the left-vague/right-weak topology of the eigenvalues point process stated in Theorem \ref{thm_main}.

\medskip

\noindent To conclude the proof of Theorem \ref{thm_main}, we devote the rest of the paper to the proof of Proposition \ref{prop_cv_expl}, also using the unique identification of the finite-marginals and the tightness of the family of measures. In Section \ref{sec_6_expl}, we control the first explosion time of the diffusion $q_\mu^\beta$ in (\ref{eq_an}) and deduce the weak convergence of its $k$ first explosions times. Then, in Section \ref{sec_6_tight}, we show the tightness of the family of measures $(\nu^\beta_\mu)_{\beta >0}$. 

\medskip

\noindent Unless specified otherwise, the limits and $o$ now pertain to the asymptotics $\beta \rightarrow 0$. For lighter notations, we also omit the dependency on $\beta$ of our variables.  

\subsection{Useful concepts and results}

We will use the following estimates for any Brownian motion $B$:
\begin{gather}\label{eq_Browniantail}
\forall x \geqslant 0, \ \mathbb{P}\Big(\underset{s \in [0,1]}\sup B(s) >x \Big) \leqslant \mathbb{P}\Big(\underset{s \in [0,1]}\sup \big\lvert B(s) \big\rvert >x \Big)\leqslant 4e^{-x^2/2}, \\ \label{eq_Browniantail2}
\forall x \geqslant 0, \ \mathbb{P}\Big(\underset{s \in [0,1]}\sup B(s) >x \Big) = \mathbb{P}\Big(\big\lvert B(1) \big\rvert >x \Big)\geqslant 1-\sqrt{\frac{2}{\pi}}x.
\end{gather} 

\noindent Consider a diffusion $y$ started from $0$ and the same diffusion $\overline{y}$ reflected downwards at the origin:
\begin{equation*}
\overline{y}(t) = y(t)-\ \underset{s \leqslant t}\sup \ y(s).
\end{equation*}
For all $\delta>0$, $\big\lbrace \overline{y}(t)<-\delta \big\rbrace = \big\lbrace \exists s<t, \ y(t)-y(s)<-\delta \big\rbrace$, therefore:
\begin{equation}\label{eq_drawdown}
\underset{s \in [0,t]}\sup\big\lvert y(s) \big\rvert < \delta/2 \Rightarrow \underset{s \in [0,t]}\inf \overline{y}(s) >-\delta.
\end{equation}

%For a continuous real-valued semimartingale $Z(t)$, $t \in \mathbb{R}^{+}$, we define its local time as the following two-parameter random function, where \textit{Leb} stands for the Lebesgue measure on $\mathbb{R}$ (see \citep{BrownianMotionAndStochasticCalculus}) :
% 
%\begin{align}\label{eq_loctime}
%L_t^x(Z) =  \lim_{\epsilon \downarrow 0} \frac{1}{4\epsilon} \mathit{Leb}\lbrace 0 \leqslant s \leqslant t, \lvert Z(t)-x \rvert \leqslant \epsilon \rbrace, \ t \in \mathbb{R}^+, \ x \in \mathbb{R}.
%\end{align}

\section{Control of the explosion times}\label{sec_6_expl}

In this section, we fix $\mu >0$. Recall the definition of the critical line $c = c_\mu$ in \eqref{criticalline}. Consider the first two explosion times $\xi^+ := \xi_\beta^+(0)$ and $\xi^- := \xi_\beta^-(0)$ of the diffusion $q_\mu^\beta$. Until the first explosion time $\xi^+$, by definition $q_\mu^\beta(0) = \infty$ and $q_\mu^\beta(t) = q^+(t)$ follows the SDE \eqref{SDEq+}.

\noindent Set $\delta=\beta^{1/8}$ and introduce the first hitting times by the diffusion $q^+$:
\begin{equation*}
\tau_0 := \inf\big\{t\geqslant 0,\; q^+(t) \leqslant 0\big\} \text{ and }\tau_c :=  \inf\big\{t\geqslant 0,\; q^+(t) \leqslant c(t)+\delta\big\}.
\end{equation*}

\noindent We decompose the trajectory of $q^+$ into three parts. First, it reaches the axis $x= 0$ in a short time (descent from $\infty$). Then, it spends a time of order $O(1)$ in the region $[c(t)+\delta,0]$ and behaves like $r^+$, the reflected Brownian motion with drift $a/4$ from \eqref{eq_r+}. Finally, if it approaches the critical line $t \mapsto c(t)$ closer than $\delta$, then it explodes with high probability within a short time (explosion to $-\infty$).

\noindent Recall the first hitting times $\xi_0^+(0)<\xi_0^-(0)$ of the critical line $c$ by the diffusion $r_\mu$ from Definition \ref{def_0_limdif}.

\begin{proposition}[Limit behavior of the diffusions $q^+$ and $q^-$]\label{prop_an}~\\
Set $T>0$, independent of $\beta$. There exist a deterministic function $\eta(\beta)\rightarrow 0$ and an event $\mathcal{E}_0$, $\mathbb{P}(\mathcal{E}_0)\rightarrow 1$, on which, for $\beta$ small enough:
\begin{enumerate}[label=(\alph*)]
	\item $\tau_0 < \beta$, 
	\item $\underset{[\tau_0, \tau_c\wedge T]}\sup\big\lvert q^+(t)-r^+(t)\big\rvert <\delta$, 
	\item $\tau_c<T \Rightarrow \lvert\xi^+-\tau_c\rvert<\eta$ and $\lvert\xi_0^+(0)-\tau_c\rvert<\eta$.
\end{enumerate}

\noindent Properties (a), (b) and (c) also hold for the diffusions $q^-$ from \eqref{SDEq-} and $r^-$ from \eqref{eq_r-}, with their corresponding hitting and explosion times. As a consequence, on the event $\mathcal{E}_0$, for $\beta$ small enough,
\begin{equation}\label{eq_an}
\xi^-<T \Rightarrow \big\lvert \xi^- -\xi_0^-(0)\big\rvert <2\eta.
\end{equation}
\end{proposition}

\noindent The control (\ref{eq_an}) extends to any $\xi_\beta^-(k)$ and $\xi_0^{-}(k)$ for $k \in \mathbb{N}$ and ensures that, for any $T>0$, $\mathbb{P}(\xi_\beta^-(k)\leqslant T) \rightarrow \mathbb{P}(\xi_0^-(k)\leqslant T)$, thus identifying the measure $\nu_\mu^0$ as the unique possible limit for $\nu_\mu^\beta$. 

\medskip

\noindent The rest of this section is dedicated to the proof of Proposition \ref{prop_an}. Recall that the diffusion $q^-$ differs from its counterpart $q^+$ only by its constant drift component $-(a+1)/4$ (instead of $+a/4$ for $q^+$), which makes $q^-$ decrease faster than $q^+$. We prove the results for the diffusion $q^+$, they extend to the diffusion $q^-$ with the same arguments.

\medskip
 
\noindent We introduce the stationary diffusion $\overline{q}$ on $\mathbb{R}_+$, which we use to approximate $q^+$ in the region where the drift component $\exp\big[-\frac{1}{\beta}\big(q^+(t)-c(t)\big)\big]$ becomes negligible as $\beta$ tends to $0$:
\begin{equation}\label{eq_an_stat}
\mathrm{d} \overline{q}(t)= \mathrm{d} W(t) +\frac{1}{4}\big(a-e^{\overline{q}(t)/\beta}\big)\mathrm{d} t.\\
\end{equation}

\subsection{Descent from $\infty$: proof of $(a)$}

It suffices to prove property $(a)$ for the diffusion $\overline{q}$, which bounds the diffusion $q^+$ from above. Set the level $l_1 := \beta^{3/4}$, so that $\beta=o(l_1)$. As $\beta$ tends to $0$, when the diffusion $\overline{q}$ is above the level $l_1$, the term of leading order in the right-hand side of (\ref{eq_an_stat}) is $e^{\overline{q}(t)/\beta}$.  

\medskip

\noindent Consider the ordinary differential equation  on $\mathbb{R}_+$:
\begin{equation*}\label{eq_an_descent_det}
\mathrm{d} y(t) := \frac{1}{4}(a-e^{y(t)/(2\beta)})\mathrm{d} t, \ y(0)=\infty,
\end{equation*}
which has for solution $y(t) = -2\beta \ln\big(\frac{1}{a}(1-e^{-at/(8\beta)})\big)$. The time $t_1$ at which $y$ reaches the level $l_1/2$ has the asymptotics
\begin{equation*}\label{eq_an_descent_det_time}
t_1 = 8\beta e^{-l_1/(4\beta)} + o(\beta e^{-l_1/(4\beta)}).
\end{equation*}
Introduce the diffusion $\overline{q}_1(t):= \overline{q}(t)-W(t)$. Its evolution writes:
\begin{equation*}
\mathrm{d} \overline{q}_1(t)=\frac{1}{4}\Big(a-e^{\big(\overline{q}_1(t)+W(t)\big)/\beta}\Big)\mathrm{d} t.
\end{equation*}
Let $\mathcal{E}'_1:=\big\lbrace \underset{[0,t_1]}\sup \big\lvert W(t) \big\rvert \leqslant \beta^2\big\rbrace$. By the Brownian tail bound (\ref{eq_Browniantail}), $\mathbb{P}(\mathcal{E}'_1) \longrightarrow 1$. On the event $\mathcal{E}'_1$ and while $\overline{q}_1(t) \geqslant l_1/2$, we have 
\begin{equation*}
\left| \frac{W(t)}{\overline{q}_1(t)} \right| \leqslant \frac{\beta^2}{l_1/2}\,, \ \text{  i.e } \ \left| \frac{W(t)}{\overline{q}_1(t)} \right| \leqslant 2\beta^{5/4}
\end{equation*}
and thus $\overline{q}_1(t) + W(t) \geqslant \overline{q}_1(t)/2$ for $\beta$ small enough.

\noindent Therefore the diffusion $\overline{q}_1$ is bounded from above by $y$ for $\beta$ small enough and hits the level $l_1/2$ before time $t_1$. Since $\big\lvert \overline{q}_1(t)-\overline{q}(t) \big\rvert \leqslant \beta^2$ before time $t_1$, for $\beta$ small enough, the diffusion $\overline{q}$ hits the level $l_1$ before time $t_1$ .

\medskip

\noindent After the level $l_1$ is reached, we use the Brownian motion $W(t_1+\cdot)-W(t_1)$ to reach $x=0$ in a short additional time. Set the event 
\begin{equation*}
\mathcal{E}^{''}_1 := \Big\lbrace \underset{[0,\beta/2]}\inf\Big\lbrace W(t_1+t)-W(t_1)+\frac{a}{4}t \Big\rbrace<-l_1 \Big\rbrace, 
\end{equation*}
on which $\tau_0<t_1+\beta/2$. Since $\mathbb{P}\big(\mathcal{E}^{''}_1\big)\geqslant \mathbb{P}\Big(\underset{[0,\beta/2]}\inf W(t_1+t)-W(t)<-l_1-\frac{a}{8}\beta\Big)$, the lower bound (\ref{eq_Browniantail2}) and the asymptotics $\beta=o(l_1)$ and $l_1=o(\sqrt{\beta})$ imply that $\mathbb{P}\big(\mathcal{E}^{''}_1\big) \longrightarrow 1$, thus proving the property $(a)$ on the event $\mathcal{E}_1=\mathcal{E}'_1 \cap \mathcal{E}^{''}_1$.

\subsection{Convergence to $r^+$: proof of $(b)$}

The bound (\ref{eq_Browniantail}) shows that the probability of the following event tends to $1$ as $\beta$ tends to $0$:
\begin{equation*}
\mathcal{E}'_2=\big\lbrace \underset{t\in [0,\beta]}\sup\lvert W(t)\rvert \leqslant \beta^{1/4} \big\rbrace.
\end{equation*}
Recall that $\delta = \beta^{1/8}$ so $\beta^{1/4}=o(\delta)$, and note that the diffusion $r^+(\tau_0+t)-r^+(\tau_0)$ is equal in law to the diffusion $r^+$ (as in (\ref{eq_r+})), by the strong Markov property. Thus, to prove property $(b)$, it suffices to show that, with overwhelming probability as $\beta$ tends to $0$, for $\beta$ small enough,
\begin{equation} \label{eq_an_tube}
\underset{[0, \tau'_c\wedge T]}\sup\big\lvert q_0^+(t)-r^+(t)\big\rvert <\delta/2,
\end{equation}
where $q_0^+$ denotes the diffusion $q^+$ started from $x=0$ at time $t=0$ and $\tau'_c$ its first hitting time of $c(t)+\delta$. We write $\overline{q}_0$ the stationary diffusion $\overline{q}$ from (\ref{eq_an_stat}) started from $x=0$ at time $t=0$. For $t \in [0,\tau'_c \wedge T]$, $q_0^+(t) \geqslant c(t)+\delta$, so we have the bounds:
\begin{equation*}
\forall t \in [0,\tau'_c \wedge T], \ \overline{q}_0(t)-e^{-\delta/\beta}T\leqslant q_0^+(t) \leqslant \overline{q}_0(t).
\end{equation*}
Since $e^{-\delta/\beta}=o\big(\delta\big)$, to prove property (\ref{eq_an_tube}), it is enough to show that, on an event $\mathcal{E}'_2$ of probability going to $1$ as $\beta$ tends to $0$, for $\beta$ small enough,
\begin{equation}\label{eq_an_tube_stat}
\underset{[0, \tau'_c\wedge T]}\sup\big\lvert \overline{q}_0(t)-r^+(t)\big\rvert <\delta/4.
\end{equation}
To that end, we bound the diffusion $\overline{q}_0(t)$ from below and above by two reflected diffusions $r^+_1$ and $r^+_2$ that converge to $r^+$ as $\beta$ tends to $0$. 

\medskip

\noindent We set the level $l_2:=\beta^{1/6}$, so that $l_2=o(\delta) \text{ and } \delta =o(\sqrt{l_2})$.

\subsubsection{Lower bound}

% As in (\ref{eq_loctime}), for any continuous semimartingale $Z$, we denote by $L_t^{x}(Z)$ its local time at position $x$ and time $t$.
%L_t^0\Big(W(t)+at/4 -e^{-\frac{l_2}{\beta}}T/4\Big)

We set the level $l_2:=\beta^{1/6}$.  Let $r_1^+$ be the following diffusion, reflected downwards at the barrier $-l_2$:
\begin{equation*}
r_1^+(t) := -l_2 + W(t)+at/4 -e^{-\frac{l_2}{\beta}}T/4 - \sup_{s \leqslant t} \big(W(s)+as/4 -e^{-\frac{l_2}{\beta}}T/4\big)\vee 0.
\end{equation*}
Since the element of drift  $-e^{\overline{q}_0(t)/\beta}$ decreases when $\overline{q}_0(t)$ goes through  $\big]-\infty,-l_2\big]$, we have the lower bound:
\begin{equation}\label{eq_an_tube_low}
\forall t \in [0,\tau'_c \wedge T], \ \overline{q}_0(t)\geqslant r_1^+(t).
\end{equation}
It is straightforward that
\begin{equation*}
\forall t \in [0,\tau'_c \wedge T], \ r^+(t)-e^{-\frac{l_2}{\beta}}T/4-l_2 \leqslant r_1^+(t)\leqslant r^+(t)-l_2.
\end{equation*}
Since $e^{-\frac{l_2}{\beta}} = o(l_2)$, for $\beta$ small enough,
\begin{equation*}
\underset{[0, \tau'_c\wedge T]}\sup\big\lvert r^+(t)-r_1^+(t)\big\rvert <2l_2.
\end{equation*}

\subsubsection{Upper bound}

We wish to bound the diffusion $\overline{q}_0(t)$ from above by the diffusion $r^+(t) + l_2$. To prove that this upper bound holds with high probability as $\beta$ tends to $0$, we use the following result, that shows how unlikely it becomes for the diffusion $\overline{q}_0(t)$ to hit the level $l_2$ before any negative level.

\begin{lemma}[Levels hitting times for the diffusion $\overline{q}_0(t)$] \label{lem_an_tube_hitting}~\\
For any $\gamma<0$,
\begin{equation*}
\mathbb{P}\Big(\inf\big\lbrace t \geqslant 0, \ \overline{q}_0(t)=l_2 \big\rbrace<\inf\big\lbrace t \geqslant 0, \ \overline{q}_0(t)=\gamma \big\rbrace\Big) \longrightarrow 0.
\end{equation*}
\end{lemma}

\noindent Lemma \ref{lem_an_tube_hitting} is proved in the Appendix using standard tools of diffusion analysis.
%la référence de l'annexe sera changée après intégration au reste du manuscrit

\medskip

\noindent The choice of level $\gamma=-c(T)$ in Lemma \ref{lem_an_tube_hitting} provides the existence of an event $\mathcal{E}'_2$ of probability going to $1$ as $\beta$ tends to $0$ on which the diffusion $\overline{q}_0(t)$ hits the barrier $c(t)$ before the level $l_2$, and thus:
\begin{equation} \label{eq_an_tube_up}
\forall t \ \in [0, \tau'_c\wedge T], \ \overline{q}_0(t)\leqslant r^+(t)+l_2.
\end{equation}

\subsubsection{Conclusion}

\noindent Gathering (\ref{eq_an_tube_low}) and (\ref{eq_an_tube_up}), we get that, on $\mathcal{E}'_2$, for $\beta$ small enough,
\begin{equation*}
\forall t \ \in [0, \tau'_c\wedge T], r_1^+(t) \leqslant \overline{q}_0(t) \leqslant r^+(t)+l_2,
\end{equation*}
which implies $\underset{[0, \tau'_c\wedge T]}\sup\big\lvert \overline{q}_0(t) -r^+(t)\big\rvert < 2l_2$.
This in turn implies (\ref{eq_an_tube_stat}) and thus proves property $(b)$.

\subsection{Explosion to $-\infty$: proof of $(c)$}

We denote by $q_{+\delta}^+$ (resp. $q_{-\delta}^+$) the diffusion $q^+$ started at time $t=0$ from position $-\mu+\delta$ (resp. $-\mu-\delta$). We introduce the first hitting time $\tau_\delta$ of the level $c(t)-\delta$ by the diffusion $q_{+\delta}^+$, and the explosion time $\tau_{-\infty}$ of the diffusion $q_{-\delta}^+$ to $-\infty$.

\noindent Recall that $l_2=\beta^{1/6}$. To prove property $(c)$, we choose $\eta:=2l_2$ and show that there exists an event of probability going to $1$ as $\beta$ tends to $0$ on which, for $\beta$ small enough, $\tau_\delta<\eta/2$ and $\tau_{-\infty}<\eta/2$.

\subsubsection{Control of $\tau_\delta$} 

Recall that $\delta = \beta^{1/8}$, so that $l_2=o(\delta)$ and $\delta=o(\sqrt{l_2})$.

\noindent We use the variations of the Brownian motion $W$ to cross the critical line $c(t)$. The upper bound $q_{\delta}^+(t)\leqslant -\mu +\delta+W(t)+\frac{a}{4}t$ implies that $\tau_\delta < l_2$ on the event
\begin{equation*}
\mathcal{E}_3 := \bigg\lbrace \underset{[0,l_2]}\inf\Big\lbrace W(s)+\frac{a}{4}s + \frac{1}{4}s +2\delta \Big\rbrace<0 \bigg\rbrace,
\end{equation*}
and the Brownian tail bound from (\ref{eq_Browniantail2}) shows that $\mathbb{P}\big(\mathcal{E}_3\big) \rightarrow 1$.

\medskip

\noindent Note that the inclusion of events $\mathcal{E}'_3 \subset \Big\lbrace \inf\big\lbrace t \geqslant 0, \ r^+(t)+t/4\leqslant-2\delta \big\rbrace< l_2\Big\rbrace$ ensures that, on a subevent of $\mathcal{E}_2$ (where $\lvert q^+(\tau_c)-r^+(\tau_c)\rvert<\delta$ if $\tau_c<\infty$) of probability going to $\mathbb{P}(\mathcal{E}_2)$ as $\beta$ tends to $0$, the diffusion $r^+$ hits the critical line $c(t)$ while the diffusion $q^+$ crosses this line, between times $\tau_c$ and $\tau_c+\eta$.

\subsubsection{Control of $\tau_{-\infty}$}

On each Brownian trajectory, the diffusion $q_{-\delta}^+$ is bounded from above by the diffusion $z$, with
\begin{equation*}
\mathrm{d} z(t):= \mathrm{d} W(t) +\frac{1}{4}\big(a-e^{-\frac{1}{\beta}\big(t/4+\mu+z(t)\big)}\big)\mathrm{d} t, \ z(0)=-\mu-\delta.
\end{equation*}
Define the diffusion $z_1(t)=z(t)+\mu-W(t)-at/4$, with evolution
\begin{equation*}
\mathrm{d} z_1(t):= -\frac{1}{4}e^{-\frac{1}{\beta}\big(z_1(t)+t/4+W(t)+at/4\big)}\mathrm{d} t, \ z_1(0)=-\delta.
\end{equation*}
Consider the event $\mathcal{E}_3^{'}:= \big\lbrace \underset{[0,\beta]}\sup \ \lvert W(t)\rvert \leqslant \delta/4 \big\rbrace$, with $\mathbb{P}\big(\mathcal{E}_3^{'} \big)\rightarrow 1$. 

\noindent On $\mathcal{E}_3^{'}$, for $\beta$ small enough:
\begin{equation*}
\forall t\leqslant \beta,\ W(t)+\frac{a+1}{4}t \leqslant \delta/2,
\end{equation*}
so the diffusion $z_1$ is bounded from above by the solution $z_2$ of the ordinary differential equation
\begin{equation*}
\mathrm{d} z_2(t):= -\frac{1}{4}e^{-\frac{1}{\beta}\big(z_2(t)+\delta/2\big)}\mathrm{d} t, \ z_2(0):=-\delta
\end{equation*}
with solution
\begin{equation*}
z_2(t)=-\delta/2+\beta\ln\big(e^{-\delta/(2\beta)}-\frac{t}{4\beta}\big),
\end{equation*}
which explodes to $-\infty$ in a time $4\beta e^{-\delta/(2\beta)}$, smaller than $\beta$ for $\beta$ small enough. This remains true for $z_1$ and thus for $z$, since $\lvert z_1-z\rvert \leqslant\mu+a\beta/4+\beta$ while $t\leqslant \beta$. Since $\beta =o(\eta)$, this proves the desired control on $\tau_{-\infty}$.

\section{Tightness of the explosion times measures}\label{sec_6_tight}

\noindent In this section, we fix $\mu >0$. Recall the measure of the explosion times $\nu_\mu^\beta$ from \eqref{explosionsq-} . We prove in this section that there are $\beta_0>0$ and $\alpha>0$, such that, for all $\epsilon>0$, there exist a finite time $T_\epsilon$ and a finite number of explosions $N_\epsilon$ so that:
\begin{equation}\label{eq_tight_main0}
\underset{\beta \leqslant \beta_0}\inf \mathbb{P}\bigg(\Big\lbrace\nu_\mu^\beta\big([0,\alpha T_\epsilon]\big)\leqslant N_\epsilon \Big\rbrace \bigcap \Big\lbrace \nu_\mu^\beta\big([\alpha T_\epsilon,\infty[\big)=0\Big\rbrace\bigg)>1-\epsilon.
 \end{equation}

\noindent Introduce $\mathcal{L_\mu^\beta}$, the law of the random measure $\nu_\mu^\beta$. The bound (\ref{eq_tight_main0}) gives us the tightness condition:
\begin{equation}\label{eq_tight_compact}
\exists \beta_0, \ \forall \epsilon>0, \ \exists K_\epsilon \text{ compact}, \ \underset{\beta<\beta_0}\sup \ \mathcal{L_\mu^\beta}(K_\epsilon)\geqslant 1-\epsilon.  
\end{equation}

\noindent  Indeed, the set $\mathring{K}_\epsilon :=\big\lbrace \mu \in \mathcal{P}, \ \mu([0,\alpha T_\epsilon]) \leqslant N_\epsilon \text{ and } \mu\big([\alpha T_\epsilon,\infty[\big)=0\big\rbrace$, where $\mathcal{P}$ is the space of locally finite measures on $\mathbb{R}_+$, satisfies the following conditions:
\begin{equation}\label{eq_tightcompact2}
\begin{split}
\underset{\mu \in \mathring{K}_\epsilon}\sup \mu(\mathbb{R_+})<\infty, \\
\underset{t>0}\inf \ \underset{\mu \in \mathring{K}_\epsilon}\sup \mu\big([t,\infty[\big)=0,
\end{split}
\end{equation}

\noindent where the second condition holds since, for any measure $\mu$ in $\mathring{K}_\epsilon$ and $t\geqslant \alpha T_\epsilon, \ \mu \big([t,\infty[\big)=0$. It thus fulfills the Kallenberg criterion for weak relative compactness (see \cite{RandomMeasures}), and its compact closure $K_\epsilon$ verifies the tightness condition (\ref{eq_tight_compact}).

\noindent Prokhorov's theorem gives us the relative compactness of the family $(\mathcal{L_\mu^\beta})_{\beta<\beta_0}$ in $\mathcal{M}^1(\mathcal{P})$, the set of probability measures on $\mathcal{P}$, and thus concludes the proof of Proposition \ref{prop_cv_expl}.
\medskip

\noindent The rest of this section is dedicated to the proof of (\ref{eq_tight_main0}). We first show a preliminary result that will be helpful to control the number of explosions. Recall that $\xi^+=\xi_\beta^+(0)$ is the first explosion time of the diffusion $q^+$ from \eqref{SDEq+}, started from $\infty$ at time $0$.

\begin{lemma}[Lower bound on the explosion time of $q^+$]\label{lem_ctrl} ~\\
For $\beta$ small enough,
\begin{equation*}
\forall t>0, \ \mathbb{P}\big(\xi^+>t\big) \geqslant 1 -4e^{-\mu^2/(32t)}.
\end{equation*}
\end{lemma}

\begin{proof} We fix a deterministic $\delta_0$ such that $0<\delta_0<\mu/4$. Recall the definition of the critical line $c_\mu$ in \eqref{criticalline}. When the diffusion $q^+$ is in the region between $-c_\mu(t)+\delta_0$ and $-\delta_0$, we have the lower bound, for $\beta$ small enough:
\begin{equation}\label{eq_ctrl_bound}
a-e^{q^+(t)/\beta}-e^{-\frac{1}{\beta}(\mu(t)+q^+(t))} \geqslant a-2e^{-\delta_0/\beta}\geqslant 0.
\end{equation}
Introduce the diffusion $\widehat{q}$ on $\mathbb{R}_+$, defined as the Brownian motion $W(t)$ started from $-\delta_0$ at time $0$ and reflected downwards at $-\delta_0$:
%L_t^{0}\big(W(t)+at/8\big)
$\widehat{q}(t) := -\delta_0+W(t)-\sup_{s \leqslant t} \big(W(s)\big)\vee 0.$
The bound (\ref{eq_ctrl_bound}) shows that, for $\beta$ small enough, the diffusion $q^+$ is bounded from below by the diffusion $\widehat{q}$ on each Brownian trajectory, up until the first hitting time of $c_\mu(t)+\delta_0$ by $q^+$. 

\medskip

\noindent Set $t>0$ and introduce the event $\mathcal{E} := \big\lbrace\underset{s \in [0, t]}\sup\big\lvert W(s) \big\rvert <\mu/4\big\rbrace$. By the Brownian tail bound (\ref{eq_Browniantail}), we have $\mathbb{P}(\mathcal{E})\geqslant 1 -4e^{-\mu^2/(32t)}$. Besides, on the event $\mathcal{E}$, using (\ref{eq_drawdown}),  
\begin{equation*}
\forall s \in [0, t], \ \widehat{q}(s) > -\delta_0-\mu/2.
\end{equation*}
This means that the diffusion $\widehat{q}$ stays above $-\mu$ until time $t$. Thus, for $\beta$ small enough so that (\ref{eq_ctrl_bound}) holds, on the event $\mathcal{E}$, we have $\xi^+>t$.
\end{proof}

We now turn to the proof of (\ref{eq_tight_main0}). Fix $\epsilon >0$. We control the diffusion $q_\mu^\beta$ with two diffusions. The first diffusion $Q_1$ starts at time $T_\epsilon$ at position $-1$ and is reflected below the horizontal line $-1$ with drift $a/8$. The second diffusion $Q_2$ starts at time $2T$ at position $c_\mu(T_\epsilon)$, has a drift $a/8$ as well and is also reflected below $-1$. 

\medskip

\noindent We can choose $T_\epsilon$ high enough such that the diffusions $Q_1$ and $Q_2$ do not hit  $c_\mu(t)+1$ with probability greater than $1-\epsilon/10$. Indeed, the sublinearity of the Brownian motion $W(T_\epsilon+t)-W(T_\epsilon)$ is such that
\begin{equation} \label{eq_Teps}
\exists \ T_0, \forall t \geqslant T_0, \lvert W(T_\epsilon+t)-W(T_\epsilon) \rvert <t/16  \text{ with probability greater than } 1-\epsilon/40.
\end{equation}
On the event where (\ref{eq_Teps}) holds, the diffusion $Q_1$ stays above $-1-t/8$ after time $T_0$, and thus above the critical line $c_\mu(t)$. We now choose $T_\epsilon$ high enough so that $\lvert W(T_\epsilon+t)-W(t) \rvert <c_\mu(T)/4$ with probability greater than $1-\epsilon/40$ until time $T_0$, thus $Q_1$ stays above $c_\mu(t)$ before time $T_0$ as well. Therefore, on an event of probability greater than $1-\epsilon/20$, the diffusion $Q_1$ never hits $c_\mu(t)$. Similar arguments can be used for the second diffusion $Q_2$.

\medskip

\noindent The term $-(a+1)/4$ in the drift of the diffusion $q^-$ implies the existence of $\alpha>0$ such that, almost surely, when started before time $2T_\epsilon+1$, the diffusion $q^-$ explodes before time $\alpha T_\epsilon$. 

\medskip

\noindent If, at time $T_\epsilon$, the diffusion $q_\mu^\beta$ evolves as the diffusion $q^-$, then it almost surely explodes before time $\alpha T_\epsilon$, after which it evolves as $q^+$ and stays above $Q_1$ (for $\beta$ small enough such that $2e^{-1/\beta}<a/8$) and does not explode anymore.
\newline If, at time $T_\epsilon$, the diffusion $q_\mu^\beta$ evolves as the diffusion $q^+$, then we distinguish between three cases:
\newline First, if the diffusion $q_\mu^\beta$ hits $-1$ between times $T_\epsilon$ and $2T_\epsilon$, then $q_\beta$ stays above $Q_1$ and therefore does not explode.
\newline Else, following the proof of property $(c)$ from Proposition \ref{prop_an} in Section \ref{sec_6_expl}, we can choose a deterministic level $\delta_1>0$ so that, if $q_\mu^\beta$ reaches $t \mapsto c_\mu(t)  + \delta_1$ between time $T_\epsilon$ and $2T_\epsilon$, then it explodes before time $2T_\epsilon +1$ with probability greater than $1-\eps/10$. After that, $q_\mu^\beta$ behaves as $q^-$ and almost surely explodes one last time before time $\alpha T_\epsilon$, as previously.
\newline Finally, if the diffusion $q_\mu^\beta$ starts above $c_\mu(T_\epsilon)+ \delta_1$ at time $T_\epsilon$ and stays in the interval $[c_{\mu}(t) + \delta_1,-1]$ for all $t \in [T_\epsilon,2T_\epsilon]$, then it is bounded from below by a Brownian motion with a positive drift $a/8$, and therefore it will be above $c_\mu(T)$ at time $2T_\epsilon$ with probability greater than $1-\epsilon/10$. In this event, $q_\mu^\beta$ stays above the diffusion $Q_2$ after time $2T_\epsilon$ and thus does not explode.

\medskip

\noindent Gathering the different cases, we thus obtain the existence of an event of probability greater than $1-\epsilon/2$ on which $q_\mu^\beta$ explodes at most once after time $T_\epsilon$ and does not explode after time $\alpha T_\epsilon$. 

\medskip

\noindent To conclude the proof of the tightness criterion (\ref{eq_tight_main0}), we apply Lemma \ref{lem_ctrl} to get the existence of a finite $N_\epsilon$ and of an event of probability greater than $1-\epsilon/2$ on which the diffusion $q_\mu^\beta$ explodes at most $N_\epsilon-1$ times before time $T_\epsilon$.

\section{Conclusion}

We proved that, in the high temperatures limit $\beta \rightarrow 0$, the properly rescaled point process of the low-lying eigenvalues of the stochastic Bessel operator converges towards a simple point process on $\mathbb{R}_+$, described using coupled SDEs. This limiting point process keeps a repulsive factor and therefore differs from the Poisson point process found by Dumaz and Labbé \cite{TheStochasticAiryOperatorAtLargeTemperature} as the high temperatures limit of the Stochastic Airy operator. Our result opens research perspectives to understand the properties of this new point process.

\newpage
\section*{Appendix: Proof of Lemma \ref{lem_an_tube_hitting}}\label{ann_C}
%\addcontentsline{toc}{chapter}{Appendices}

Recall that $l_1=\beta^{3/4}$ and $l_2=\beta^{1/6}$. In this proof we denote by $q$ the diffusion $\overline{q}_0^+$. Set $\gamma<0$, $\tau := \inf\big\lbrace t \geqslant 0, \ q_0^+(t)=l_2 \big\rbrace$ and $\tau':=\inf\big\lbrace t \geqslant 0, \ q_0^+(t)=\gamma \big\rbrace$. Let $f_\beta(x):=\frac{1}{4}\big(a-\exp(x/\beta)\big)$. To compute the hitting times of $q$, introduce the scale functions $\mathfrak{s}_\beta$ and $\mathfrak{s}$:
\begin{align*}
\mathfrak{s}_\beta(x) &:= \int_{-1}^x\exp\Big(-2\int_0^u f_\beta(v)\mathrm{d} v\Big)\mathrm{d} u, \\
\mathfrak{s}(x) &:= \int_{-1}^x\exp\Big(-2\int_0^u \frac{a}{4}\mathrm{d} v\Big)\mathrm{d} u=\frac{2}{a}\big(e^{a/2}-e^{ax/2 }\big).
\end{align*}
The following lemma explicits the asymptotic behavior of $\mathfrak{s}_\beta$.

\begin{lemma}[Convergence of the scale functions] \label{lem_an_tube_scale}~\\
For any $x_0<0$, 
\begin{equation*}
\mathfrak{s}'_\beta\longrightarrow \mathfrak{s}' \text{ and } \mathfrak{s}_\beta \longrightarrow \mathfrak{s} \text{ uniformly on }[x_0, 0].
\end{equation*}
Furthermore, $\mathfrak{s}_\beta\big(l_2\big)\longrightarrow \infty$.
\end{lemma}

\noindent Since $\mathfrak{s}_\beta\big(q(\cdot)\big)$ is a local martingale, $\mathfrak{s}_\beta\big(q(\cdot\wedge \tau(q) \wedge \tau'(q)\big)$ is a martingale. By the stopping theorem, we get:
\begin{equation*}
\mathbb{P}\big(\tau<\tau'\big) = \frac{\mathfrak{s}_\beta(\gamma)-\mathfrak{s}_\beta\big(0\big)}{\mathfrak{s}_\beta(\gamma)-\mathfrak{s}_\beta\big(l_2\big)}.
\end{equation*}
Lemma \ref{lem_an_tube_scale} readily implies that this probability tends to $0$ as $\beta$ tends to $0$.

\begin{proof}[Proof of Lemma \ref{lem_an_tube_scale}]
We have, for all $x \leqslant 0$,
\begin{equation*}
\bigg\lvert \int_0^x f_\beta(v)-\frac{a}{4}\mathrm{d} v \bigg\rvert \leqslant \int_x^0 e^{v/\beta}\mathrm{d} v \leqslant \beta  \longrightarrow 0,
\end{equation*}
which means that $\ln \mathfrak{s}'_\beta$ converges uniformly to $\ln \mathfrak{s}'$ on $]-\infty,0]$.
Besides, the functions $\ln \mathfrak{s}'_\beta$ and $\ln \mathfrak{s}'$ are bounded on $[x_0, 0]$, and $x\mapsto e^x$ is uniformly continuous on $[x_0, 0]$ so $\mathfrak{s}'_\beta$ converges uniformly to $\mathfrak{s}'$ on $[x_0,0]$. Therefore, $\mathfrak{s}_\beta$ converges uniformly to $\mathfrak{s}$ on $[x_0,0]$.

\noindent Now turning to $\mathfrak{s}_\beta\big(l_2\big)$. We can compute explicitly:
\begin{equation*} \label{eq_an_tube_scale}
\ln \mathfrak{s}_\beta^{'} (x) = -2\int_0^x f_\beta(v)\mathrm{d} v = -2ax+2\int_0^x e^{v/\beta}\mathrm{d} v = -2ax + 2\beta\big(e^{x/\beta}-1\big).
\end{equation*}
A study of the variations of $\ln \mathfrak{s}'_\beta$ shows that $\mathfrak{s}'_\beta$ decreases before $\beta\ln(2a)$ and increases afterwards.
For $\beta$ small enough, $\ln \mathfrak{s}'_\beta(l_1)\geqslant \beta e^{l_1/\beta}$, which tends to $\infty$ as $\beta$ tends to $0$, and $\mathfrak{s}'_\beta$ increases on $\big[l_1,l_2\big]$, so that:
\begin{equation*}
\mathfrak{s}_\beta\big(l_2\big)\geqslant \int_{l_1}^{l_2}\mathfrak{s}'_\beta(v)\mathrm{d} v \geqslant \big(l_2-l_1\big)\mathfrak{s}'_\beta(l_1),
\end{equation*}
therefore $\mathfrak{s}_\beta\big(l_2\big)\longrightarrow \infty$.
\end{proof}

%%%%%%%%%%Declarations%%%%%%%%%%

\acks % Place the text of your acknowledgements after the \acks (or \Acks) command. This will generate the heading "Acknowledgements". If you wish to make only one acknowledgement, use \ack (or \Ack).
\noindent I wish to thank Laure Dumaz for her precious help throughout this research project.

\fund % Place any funding information for this work after the \fund (or \Fund) command.
\noindent This research paper presents results found during my PhD thesis \cite{Magaldi}, funded with a research grant from the SDOSE Doctoral School and PSL University.

\competing % Place any information on competing interests after the \competing (or \Competing) command.
\noindent There were no competing interests to declare which arose during the preparation or publication process of this article.

%%%%%%%%%%%%Reference list%%%%%%%%%%%%%%

%\newpage
\bibliographystyle{APT}
\bibliography{biblio}

%\addcontentsline{toc}{chapter}{Bibliography}

\end{document}